\documentclass[11pt,reqno]{amsart}
\usepackage[margin=3.5cm]{geometry}
\usepackage{xcolor}
\usepackage{amsfonts,amsthm,amssymb,verbatim,amscd,amsmath,blindtext}
\usepackage{multirow}
\usepackage{graphicx}
\usepackage{enumitem}
\usepackage{amssymb}
\usepackage{ytableau}
\usepackage{tikz}
\usepackage{tikz-cd}
\usepackage{float,pifont}
\usepackage{mathtools}
\usepackage[pagebackref,colorlinks=true,citecolor=blue]{hyperref}
\usepackage[capitalise]{cleveref}

\theoremstyle{plain}
\newtheorem{theorem}{Theorem}[section]
\newtheorem{lemma}[theorem]{Lemma}

\newtheorem{claim}[theorem]{Claim}
\theoremstyle{definition}
\newtheorem{remark}[theorem]{Remark}
\newtheorem{notation}[theorem]{Notation}
\newtheorem{question}[theorem]{Question}

\newcommand{\ass}{\operatorname{Ass}}

\newcommand{\vv}{\operatorname{v}}
\newcommand{\p}{\mathfrak{p}}
\newcommand{\q}{\mathfrak{q}}
\newcommand{\kk}{\mathbb{K}}
\newcommand{\ini}{\operatorname{in}}

\begin{document}

\title{The v-numbers of permanental ideals}

\author{Trung Chau}
\email{chauchitrung1996@gmail.com}
\address{Chennai Mathematical Institute, Siruseri, Tamil Nadu, India}

\author{A. V. Jayanthan}
\email{jayanav@iitm.ac.in}
\address{Department of Mathematics, Indian Institute of Technology Madras, Chennai, Tamil Nadu, India - 600036.}
\begin{abstract}
In this article, we compute the $\vv$-number of $2\times 2$ permanental ideals of generic, generic symmetric, and generic Hankel matrices.
\end{abstract}



\maketitle


\section{Introduction}
Let $\kk$ be a field and $R = \kk[x_1,\ldots,x_n]$ be a standard graded polynomial ring and $I$ be a graded ideal in $R$. In \cite{CSTPV20}, Cooper et. al. introduced the notion of $\vv$-number of $I$, denoted by $\vv(I)$:
\[
\vv(I) \coloneqq \min\{d \mid \exists~ f \in R_d \text{ and } \p \in \ass(I) \text{ such that } I \colon f = \p\}.
\]
They introduced this invariant to study the minimum distance function of projective Reed-Muller-type codes. Since then, there has been a lot of research on the behavior of $\vv(I)$ in comparison with algebraic/combinatorial aspects of the ideal $I$, see \cite{ASS24,BM25,DJS25,Fic25,JV21,KMT25,SK24,SS22} for a non-exhaustive list.

The permanent of a square matrix is exactly its determinant with all coefficients being $+1$, instead of $\pm 1$. For a matrix $X$ with entries in a ring $R$, the \emph{$t\times t$ permanental ideal of $X$}, denoted by $P_t(X)$, is the ideal in $R$ generated by all $t\times t$ subpermanents of $X$. Their counterpart, determinantal ideals, plays an important role and is more intensively studied in geometry and algebra. Permanental ideals, on the other hand, have proved to behave erratically, e.g., their minimal primes are not fully classified (cf. \cite{BCMV25,Kirkup} for partial answers), or their Frobenius singularities are worse than those of the corresponding determinantal ideals \cite{Chau-F}. It is worth noting that permanental ideals and their determinantal counterpart coincide when the characteristic of the base ring is 2. Thus the former is an interesting class of ideals whose behaviors depend on characteristics.

In this article, we fully compute the $\vv$-number of $2\times 2$ permanental ideals of generic, generic symmetric, and generic Hankel matrices in a polynomial ring over a field $\kk$ with $\operatorname{char} \kk \neq 2$. We note that in the case $\operatorname{char} \kk =2$, the permanental ideals, which coincide with the corresponding determinantal ideals, are prime, and thus their $\vv$-numbers are $0$. The main tool that we use for the computation of $\vv$-number of these ideals is a key result in \cite{GRV21} (see \Cref{lem:upper-bound-v-number,lem:lower-bound-v-number}). Gr\"obner basis for the permanental ideals of generic, generic Hankel and generic symmetric  matrices have been computed respectively in \cite{LS00}, \cite{GGS07} and \cite{Chau}. We use these Gr\"obner bases and the fact that if $\ini(f) \notin \ini(I) \colon \ini(J)$, then $f \notin I \colon J$ to obtain required lower bounds for the $\vv$-number.

Our finding in this article can be summerized as below:
\begin{table}[h!]
\begin{tabular}{cc|c|}
\cline{3-3}
&  & v-number of $P_2(X)$ \\ \hline
\multicolumn{1}{|c|}{\multirow{3}{*}{$X$ is $m\times n$ generic}} & $(m,n)=(2,2)$ & $0$   \\ \cline{2-3} 
\multicolumn{1}{|c|}{}                             & $m=2$ and $n\geq 2$ & 2 \\ \cline{2-3} 
\multicolumn{1}{|c|}{}                            & $m,n\geq 3$ & 3 \\ \hline
\multicolumn{1}{|c|}{\multirow{2}{*}{$X$ is $n\times n$ generic symmetric}} & $n=2$ & $0$ \\ \cline{2-3} 
\multicolumn{1}{|c|}{}
& $n\geq 3$ & 3 \\ \hline
\multicolumn{1}{|c|}{\multirow{6}{*}{$X$ is $m\times n$ generic Hankel}} & $(m,n)=(2,2)$ & $0$ \\ \cline{2-3} 
\multicolumn{1}{|c|}{}                                                      & $m\geq 3$ and $m+n\geq 10$       & \multirow{2}{*}{1}   \\ \cline{2-2}
\multicolumn{1}{|c|}{}                                                      & $(m,n)\in \{(3,6),(4,5)\}$       &                      \\ \cline{2-3} 
\multicolumn{1}{|c|}{}                                                      & $m=2$ and $n\geq 4$              & \multirow{2}{*}{2}   \\ \cline{2-2}
\multicolumn{1}{|c|}{}                                                      & $(m,n)\in \{(3,4),(3,5),(4,4)\}$ &                      \\ \cline{2-3} 
\multicolumn{1}{|c|}{}                                                      & $(m,n)\in \{(2,3),(3,3)\}$       & 3                    \\ \hline
\end{tabular}
\end{table}

We note that in the case $(m,n)=(2,2)$, the ideal $P_2(X)$ is known to be prime in all cases where $X$ is generic/generic symmetric/generic Hankel (cf. \cite[Lemma~3.2]{Chau-F}). In particular, in this case, $\vv(P_2(X))=0$ by definition. In this article we will thus ignore this case.

In the second section, we fix the notation for the rest of the paper and recall some results from the literature which are needed in our proofs. In the next three sections, we compute the $\vv$-number of permanental ideals of the generic, generic symmetric and generic Hankel matrices.

\section{Preliminaries}

In this section, we introduce/recall some notation and terminology that will aid us in later sections. 

The next lemma follows directly from the definition of $\vv$-number. We write it here for ready reference, as it is our main tool to show an upper bound for $\vv$-numbers.

 \begin{lemma}\label{lem:upper-bound-v-number}
     Let $I$ be a homogeneous ideal. If $f$ is a homogeneous polynomial such that $(I\colon f)$ is a prime ideal. Then  $\vv(I) \leq \deg(f)$.
 \end{lemma}
For a homogeneous ideal $I$ of a graded ring $A = \oplus_{n\geq 0}A_n$, set $\alpha(I) := \min\{\deg f ~\mid~ f \in I \text{ is homogeneous}\}$. The next two lemmas are simple observations which we require in later sections.

\begin{lemma}\label{lem:alpha-smaller-for-colon-of-small-ideals}
    Let $I,J,K$ be homogeneous ideals in a polynomial ring $R$. If $J\subseteq K$, then
    \[
     \alpha((I\colon K)/I)\geq \alpha((I\colon J)/I).
    \]
\end{lemma}
\begin{proof}
Since $J\subseteq K$, we have $I\colon K \subseteq I\colon J$. The result then follows from definition.
\end{proof}

\begin{lemma}\label{lem:lower-bound-v-number}
Let $I$ be a homogeneous ideal. Then
    \[
    \vv(I) \geq \min \{ \alpha((I\colon P)/I) \mid P\in \ass(I) \} = \min \{ \alpha((I\colon P)/I) \mid P\in \operatorname{Min}(I) \}.
    \]
\end{lemma}
\begin{proof}
The first inequality follows from \cite[Theorem 10 (b) and (c)]{GRV21}. The equality follows from \Cref{lem:alpha-smaller-for-colon-of-small-ideals}.
\end{proof}

The computation of colon ideals are important even just to bound the v-numbers. A common theme in this article is the use of initial ideals and Gr\"obner basis in bounding the v-numbers. For a homogeneous ideal $I$, let $I_{\leq k}$ be the subideal of $I$ generated by all  polynomials in $I$ of degree at most $k$, and $I_{\langle k\rangle}$ be the subideal of $I$ generated by all polynomials in $I$ of degree $k$, for each integer $k$.

\begin{lemma}\label{lem:Grobner-initial}
    Let $I,J$ be homogeneous ideals and $f$ a homogeneous polynomial. If $f\in I\colon J$, then $\ini(f)\in \ini(I)\colon \ini(J)$. Moreover, if $\ini(J)=(g_1,\dots, g_t)$ where $\deg g_i=\deg g_j$ for any $i,j=1,\dots ,t$, then \[
    \ini(f)\in \left( \ini(I)_{\langle \deg f + \deg g_1 \rangle} \colon \ini(J)\right)_{\langle \deg f \rangle}.\]
\end{lemma}

We now recall a result from \cite{Chau} which we require in later sections. We remark that the techniques originate earlier (\cite{LS00,GGS07}).

 \begin{lemma}[\protect{\cite[Lemma 5.5]{Chau}}]\label{lem:colon-equality}
    Let $I$ and $Q$ be ideals of $R$, $x$ an element in $R$, and $n$ a positive integer. Assume that the following holds:
    \begin{enumerate}
        \item $Q$ is $P$-primary where $P=\sqrt{Q}$ is a prime ideal;
        \item $x\notin P$;
        \item $I\subseteq Q$;
        \item $Q\subseteq (I\colon x^n)$.
    \end{enumerate}
    Then we have $Q=(I\colon x^n)$.
\end{lemma}

\section{v-number of permanental ideals of generic matrices}

Throughout this section, let $X=(x_{ij})$ be a generic $m\times n$, where $m\leq n$, matrix of indeterminates over a field $\kk$. We impose the  \emph{diagonal ordering} on the variables
\[
x_{ij} \succ x_{kl} \text{ if either } i<k \text{ or } \begin{cases}
    i=k \text{ and }\\
    j<l
\end{cases}.
\]
With respect to this ordering, the leading term of any sub-permanent of $H$ is the product of its diagonal, and thus this is often called a \emph{diagonal ordering}. Recall that the Gr\"obner basis of $P_2(X)$ with respect to this ordering was given in \cite[Theorem~3.1]{LS00}.

\begin{theorem}
Let $X$ be a $2\times n$ generic matrix, with $n \geq 3$. Then $\vv(P_2(X)) = 2$
\end{theorem}

\begin{proof}
First we show that $\vv(P_2(X)) \leq 2$. It follows from \cite[Theorem 4.1]{LS00} that $\p \coloneqq (x_{13},\ldots,x_{1n}, x_{23},\ldots,x_{2n}, x_{11}x_{22}+x_{12}x_{21})$ is an associated prime of $P_2(X)$. We claim that $\p = P_2(X)\colon x_{11}x_{22}$. Indeed, 
observe that $x_{11}x_{22} \notin \p$. Moreover, by \cite[Lemma 2.1]{LS00}, $x_{11}x_{2j}x_{1j}, x_{11}x_{22}x_{2k} \in \p$ for all $3 \leq j, k \leq n$. Hence $\p \subseteq P_2(X)\colon x_{11}x_{22}$. Therefore, using \Cref{lem:colon-equality} we can conclude that $\p = P_2(X)\colon x_{11}x_{22}$. This implies that $\vv(P_2(X)) \leq  2$ by \Cref{lem:upper-bound-v-number}, as desired. 

It remains to show that $\alpha((P_2(X) : \q)/P_2(X)) \geq 2$ for any associated prime $\q$ of $P_2(X)$. By \Cref{lem:lower-bound-v-number}, it suffices to show that $P_2(X)\colon \q$ does not contain any linear form. Suppose otherwise that  $0\neq f = \sum_i a_ix_{1i} + \sum_j b_jx_{2j} \in P_2(X)\colon \q$. By \cite[Theorem~4.1]{LS00}, any associated prime of $P_2(X)$ contains a variable. Without loss of generality, assume that $x_{11}\in P_2(X)$. Then $f \in P_2(X)\colon \q \subseteq P_2(X)\colon x_{11}$. By the Gr\"obner basis of $P_2(X)$ given in \cite[Theorem 3.1]{LS00}, we have
\[
\ini(P_2(X))_{\leq \deg f + \deg x_{11}} = \ini(P_2(X))_{\leq 2} = (x_{1i}x_{2j}\ \vert \ 1\leq i<j\leq n).
\]
Thus by \Cref{lem:Grobner-initial}, we have
\[
\ini(f)\in \left( \ini(P_2(X))_{\leq 2} \colon \ini(g)\right)_{\langle 1  \rangle} = (x_{2j}\ \vert \ 2\leq j\leq n),
\]
which implies that $a_i=0$ for any $i\in [n]$, i.e., $f = \sum_{j=2}^n b_jx_{2j}$. Since $x_{11}f\in P_2(X)$, we have $x_{11}f=\sum_{1\leq r<s\leq n}c_{rs}(x_{1r}x_{2s}+x_{1s}x_{2r})$ for some constants $c_{rs}$. Specializing $x_{11}=0$, we obtain
\[
\sum_{2\leq s\leq n}c_{1s}(x_{1s}x_{2r})+ \sum_{2\leq r<s\leq n}c_{rs}(x_{1r}x_{2s}+x_{1s}x_{2r}) = 0\times f=0.
\]
Since all the monomial summands of the above sum are distinct, we conclude that $c_{rs}=0$ for any $1\leq r<s \leq n$. In particular, this implies that $f=0$, a contradiction, as desired.
\end{proof}

We now proceed to compute the $\vv$-number of permanental ideals of larger matrices. We begin by obtaining a lower bound for the $\vv$-number.
\begin{lemma}\label{lem:alpha-generic-3}
    Assume that $m\geq 3$. Then 
    \[
    \alpha\left(\frac{\bigcap_{i=1}^m(P_2(X)\colon x_{i1})}{P_2(X)}\right) \geq 3.
    \]
\end{lemma}

\begin{proof}
Equivalently, we want to show that no element in the module $\frac{\bigcap_{i=1}^m(P_2(X)\colon x_{i1})}{P_2(X)}$ is of degree $2$. Suppose  that there exists a quadratic homogeneous polynomial $f \in \bigcap_{i=1}^m(P_2(X)\colon x_{i1})$. It now suffices to show that $f\in P_2(X)$. Set 
    \[
    f=\sum_{\substack{y,z\in X\\ y\succeq z}}a_{yz}yz.
    \]
Observe that we can replace $f$ with $f-g$ for any $g\in P_2(X)$. Therefore, we can assume that $a_{yz}=0$ for any $y,z\in X$  that form the diagonal of some $2\times 2$ submatrix of $X$.  We proceed with the following claim.

    \begin{claim}\label{clm:y-z-same-row}
        We have $a_{yz}=0$ whenever $y,z$ are in the same row (and not necessarily distinct).
    \end{claim}
    \begin{proof}
        Set $y=x_{ij}$ and $z=x_{ik}$ for some $j\leq k$. By the symmetricity of the arguments, we can assume that $i=1$, $j\in \{1,2\}$, and $k\in \{1,2,3\}$. It then suffices to show that $a_{x_{11}^2}=a_{x_{11}x_{12}}=a_{x_{12}^2}=a_{x_{12}x_{13}}=0$. Indeed, since $f\in P_2(X)\colon x_{11}$, we have $x_{11}f\in P_2(X)$. By setting all variables, except  $x_{11}$, $x_{12}$, and $x_{13}$, to $0$, we obtain 
        \[
        x_{11}\bar{f} \in \overline{P_2(X)}=(0),
        \]
        where $\bar{ ~ }$ denotes the object after substituting the desired variables to zero. Therefore 
        \[
        0 = \bar{f}= a_{x_{11}^2}x_{11}^2 + a_{x_{12}^2}x_{12}^2 + a_{x_{13}^2}x_{13}^2 + a_{x_{11}x_{12}}x_{11}x_{12} + a_{x_{11}x_{13}}x_{11}x_{13}+a_{x_{12}x_{13}}x_{12}x_{13}.
        \]
        Hence $a_{x_{11}^2}=a_{x_{11}x_{12}}=a_{x_{12}^2}=a_{x_{12}x_{13}}=0$, as desired.
    \end{proof}
    \begin{claim}\label{clm:y-z-same-column-and-antidiagonal}
        We have $a_{yz}=0$ whenever one of the following holds:
        \begin{itemize}
            \item $y,z$ are in the same column;
            \item either $y$ or $z$ is in the first column of $X$.  
        \end{itemize}
    \end{claim}
    \begin{proof}
        By similar arguments as in the proof of \cref{clm:y-z-same-row}, it suffices to show that $a_{x_{11}x_{21}}=a_{x_{12}x_{22}}=a_{x_{21}x_{12}}=0$. Indeed, since $f\in P_2(X)\colon x_{21}$, we have $x_{21}f\in P_2(X)$. Setting all variables, except $x_{11}$, $x_{12}$, $x_{21}$, and $x_{22}$, to $0$, we obtain 
        \[
        x_{21}\bar{f} \in \overline{P_2(X)}=(x_{11}x_{22}+x_{12}x_{21}).
        \]
        Observe that we already have $a_{x_{11}x_{12}}=a_{x_{21}x_{22}}=a_{x_{ij}^2}=0$, for any $1\leq i,j\leq 2$, from \cref{clm:y-z-same-row}, and $a_{x_{11}x_{22}}=0$ since these two variables form a diagonal of a submatrix of $X$.
        Thus, under our specialization 
        we have
        \[
        0=\bar{f}= a_{x_{11}x_{12}}x_{11}x_{12} +  a_{x_{21}x_{22}}x_{21}x_{22} +  a_{x_{21}x_{12}}x_{21}x_{12}.
        \]
        Observe that none of the monomial summand of $x_{21}\bar{f}$ is divisible by $x_{11}x_{22}$, the leading term of $x_{11}x_{22}+x_{12}x_{21}$. Thus $x_{21}\bar{f} \in \overline{P_2(X)} $ implies that  $a_{x_{11}x_{21}}=a_{x_{12}x_{22}}=a_{x_{21}x_{12}}=0$, as desired.
    \end{proof} 
    By \cref{clm:y-z-same-row} and \cref{clm:y-z-same-column-and-antidiagonal}, we now have $f=\sum_{y\prec z} a_{yz}yz,$ where $\{y,z\}$ ranges among the set of two entries both of which are not in the first column. 
    Observe that, by setting all variables, except the entries in the submatrix of $X$ that contains $y$, $z$, the first column, and a row that does not contain $y$ or $z$, to $0$, we may as well assume that $X$ is of size $3\times 3$. And by the similarity of arguments, it suffices to show that $a_{x_{13}x_{22}}=0$. As consequences of \cref{clm:y-z-same-row} and \cref{clm:y-z-same-column-and-antidiagonal}, $f$ reduces to the form
    \[
    f= a_{x_{13}x_{22}}x_{13}x_{22} + a_{x_{13}x_{32}}x_{13}x_{32}+ a_{x_{23}x_{32}}x_{23}x_{32}.
    \]
    Since $f\in P_2(X)\colon  x_{31}$ (the index makes sense since $m\geq 3$), we have
    \[
    a_{x_{13}x_{22}}x_{13}x_{22}x_{31} + a_{x_{13}x_{32}}x_{13}x_{32}x_{31}+ a_{x_{23}x_{32}}x_{23}x_{32}x_{31} = x_{31}f \in P_2(X).
    \]
    This implies that, if $a_{x_{13}x_{22}} \neq 0$, then $\ini(x_{31}f) = a_{x_{13}x_{22}}x_{13}x_{22}x_{31} \in \ini(P_2(X))$. 
This is a contradiction since by \cite[Theorem~3.1]{LS00} the initial ideal of $P_2(X)$ does not contain $x_{13}x_{22}x_{31}$. Therefore, we have $a_{x_{13}x_{22}}=0$, as desired.
This concludes the proof of the Lemma.    
\end{proof}
\begin{remark}
It may be noted that, due to the symmetricity of arguments, in \Cref{lem:alpha-generic-3}, if we take the intersection over all variables on any column of $X$ or if we take the intersection over all variables on any row of $X$, we get the same conclusion.
\end{remark}
Finally, we now determine the $\vv$-number of $P_2(X)$.
\begin{theorem}\label{thm:main-v-number-generic}
Assume that $m, n\geq 3$. Then $\operatorname{v}(P_2(X))=3$.
\end{theorem}
\begin{proof}
First we show that $\vv(P_2(X)) \leq 3$. For this, we claim that \[P_2(X) : x_{11}x_{12}x_{13} = (x_{ij} \mid 2\leq i\leq m, 1\leq j \leq n)\eqqcolon \q.\]
We have
\begin{itemize}
    \item $\q$ is prime;
    \item $x_{11}x_{12}x_{13}\notin \q$;
    \item $P_2(X)\subseteq \q$ (see, e.g., \cite[Theorem 5.7]{LS00});
    \item $\q \subseteq (P_2(X)\colon x_{11}x_{12}x_{13})$ (this follows from \cite[Lemma 2.1]{LS00}).
\end{itemize}
Thus our claim follows from \cref{lem:colon-equality}. In particular, we have $\vv(P_2(X)) \leq 3$.

Let $\p$ be an associated prime of $P_2(X)$. Since $3 \leq m \leq n$, it follows from \cite[Theorem 5.7]{LS00} that $P_2(X)$ contains either all variables of one row or all variables one column of $X$. Without loss of generality, assume that $P$ contains all variables in the first column, i.e., $x_{i1} \in P$ for all $i = 1, \ldots, m$. Then it follows from \Cref{lem:alpha-smaller-for-colon-of-small-ideals} and \Cref{lem:alpha-generic-3} that 
\[\alpha\left(\frac{P_2(X) \colon \p}{P_2(X)}\right) \geq \alpha\left(\frac{P_2(X) \colon (x_{11},\ldots,x_{m1})}{P_2(X)}\right) = \alpha \left(\frac{\cap_{i=1}^m P_2(X) \colon x_{i1}}{P_2(X)}\right) \geq 3.
\]
Hence by \Cref{lem:lower-bound-v-number}, we get $\vv(P_2(X)) \geq 3$. Combining with the reverse inequality, we get $\vv(P_2(X)) = 3$, as desired.
\end{proof}

\section{v-number of permanental ideals of generic symmetric matrices}

In this section, let $Y$ denote an $n\times n$ generic symmetric matrix of indeterminates over a field $\kk$, i.e., we set
\[
Y=\begin{pmatrix}
    y_{11} & y_{12} & \cdots & y_{1n}\\
    y_{12} & y_{22} & \cdots & y_{2n}\\
    \vdots & \vdots & \ddots & \vdots\\
    y_{1n} & y_{2n} & \cdots & y_{nn}
\end{pmatrix}.
\]
We impose the \emph{diagonal ordering} on these indeterminates: 
\[
y_{11}\succ y_{12} \succ \cdots \succ y_{1n} \succ y_{22} \succ y_{23} \succ \cdots \succ y_{nn}.
\]
Note that the Gr\"obner basis of $P_2(Y)$ with respect to this ordering was given in \cite{Chau}, which we shall use extensively in this section.

\begin{theorem}[\protect{\cite[Theorem~3.1]{Chau}}]\label{thm:Grobner-symmetric}
	Let $Y$ be an $n\times n$ generic symmetric matrix with indeterminates as entries. The following collection of polynomials is a reduced Gr{\"o}bner basis for $P_2(Y)$ with respect to any diagonal ordering:
	\begin{enumerate}
		\item[(1a)]  The subpermanents $y_{ii}y_{jj}+y_{ij}^2$,\  $i<j$;
		\item[(1b)]  the subpermanents $y_{ii}y_{jk}+y_{ij}y_{ik}$,\  $j<k$,\  $i\neq j,k$;
		\item[(1c)]  $y_{ij}y_{kl}$,\  $i,j,k,l$ are distinct, $i<j$, $k<l$;
		\item[(2a)]  $y_{il}y_{jl}y_{kl}$,\  $i<j<l$,\  either $k=i$, $k=j$, or $j< k< l$;
		\item[(2b)]  $y_{il}y_{jl}y_{kk}$,\  $i<j<k<l$;
		\item[(2c)]  $y_{ij}y_{ik}y_{jj}$,\  $i<j<k$;
		\item[(3a)]  $y_{ij}y_{ik}y_{il}$,\  $i<k<l$,\  either $j=l$, $j=k$, or $i< j< k$;
		\item[(3b)]  $y_{ik}y_{il}y_{jj}$,\  $i<j<k<l$;
		\item[(3c)]  $y_{ik}y_{jk}y_{jj}$,\  $i<j<k$;
		\item[(6a)]  $y_{ik}^3y_{jj}$,\  $i<j<k$;
		\item[(6b)]  $y_{ik}^2y_{jj}^2$,\ $i<j<k$.
	\end{enumerate}
\end{theorem}

For any two integers $1\leq r<s\leq n$, set 
\[
\q_{rs}(Y)= ( y_{rr}y_{ss}+y_{rs}^2, \ y_{ij} \mid 1\leq i\leq j\leq n \text{ and } (i,j)\notin \{(r,s),(r,r),(s,s)\} ).
\]
These ideals are exactly the minimal primes of $P_2(Y)$, \cite[Theorem 4.1]{Chau}. Thus we can obtain a lower bound of $\vv(P_2(Y))$ by determining  $\alpha\left(\frac{P_2(Y)\colon \q_{ij}(Y)}{P_2(Y)}\right)$ for any $i$ and $j$ and using \cref{lem:lower-bound-v-number}. In other words, we need an analog of \cref{lem:alpha-generic-3} in this case of a symmetric matrix. 
\begin{notation}\label{yyhat}
For the rest of the section, $Y$ always denotes an $n \times n$ generic symmetric matrix and $\hat{Y}$ denotes the $(n-1) \times (n-1)$ matrix obtained by deleting the first row and the first column of $Y$. In particular, $\hat{Y}$ is also generic symmetric.
\end{notation}
\begin{lemma}\label{lem:retract}
With the notation as in \Cref{yyhat},  the natural map $\phi: \kk[\hat{Y}]/P_2(\hat{Y}) \to \kk[Y]/P_2(Y)$ splits. In particular $P_2(Y)\cap \kk[\hat{Y}]=P_2(\hat{Y})$.
\end{lemma}

\begin{proof}
Define $\psi: \kk[Y] \to \kk[\hat{Y}]/P_2(\hat{Y})$ by 
\[\psi(y_{ij}) = \left\{ \begin{array}{ll}  
0  & \text{ if }  i = 1 \text{ and } \\ y_{ij} & \text{ if } i \neq 1. \end{array} \right.\]
Suppose $f = y_{ij}y_{lk} + y_{ik}y_{lj}$, where $i\leq j$ and $k\leq l$, is any generator of $P_2(Y)$. If $i \neq 1$ and $l \neq 1$, then $f$ is a generator of $P_2(\hat{Y})$. If $i = 1$ or $l = 1$, then $\psi(f) = 0$. This implies that $P_2(Y) \subseteq \ker \psi$. Therefore there exists a map $\bar{\psi} : \kk[Y]/P_2(Y) \to \kk[\hat{Y}]/P_2(\hat{Y})$ such that for any $\bar{f}$, we have $\bar{\psi}(\bar{f}) = \psi(f)$. It is straightforward to observe that $\bar{\psi} \circ \phi = id$. 
\end{proof}


We now begin with a lemma that helps us obtaining the lower bound for the $\vv$-number. Unlike in the generic symmetric case, the proof here is a bit more technical. So, we split the cases $n = 3$ and $n > 3$.
\begin{lemma}\label{lemma:n=3}
Assume that $n=3$. Then $    \alpha\left(\frac{P_2(Y)\colon \q_{n-1,n}(Y)}{P_2(Y)}\right) \geq 3$.
\end{lemma}

\begin{proof}
    Equivalently, we want to show that no element in the module $\frac{P_2(Y)\colon \q_{n-1,n}(Y)}{P_2(Y)}=\frac{P_2(Y)\colon \q_{23}(Y)}{P_2(Y)}$ is of degree 2. Suppose that there exists a quadratic homogeneous polynomial $f\in P_2(Y)\colon \q_{n-1,n}(Y)$. It now suffices to show that $f\in P_2(Y)$. Set  \[
    f=\sum_{\substack{x,z\in X\\ x\succeq z}}a_{xz}xz.
    \]
    Observe that we can replace $f$ with $f-g$ for any $g\in P_2(Y)$. Therefore, we can assume that 
    \begin{equation}\label{eq:in-P2-symmetric-n=3}
        \text{$a_{xz}=0$ for any $x,z\in Y$  that form the diagonal of some $2\times 2$ submatrix of $Y$}
    \end{equation}
    On the other hand, by Lemma~\ref{lem:Grobner-initial} and Theorem~\ref{thm:Grobner-symmetric}, and with a remark that $P_2(Y)\colon \q_{23}(Y) = P_2(Y)\colon (y_{11},y_{12},y_{13})$, we have
    \begin{equation}\label{eq:initial-symmetric-n=3}
        \begin{multlined}
             \ini(f) \in \left( \ini( P_2(Y))_{\langle 3 \rangle} \colon (y_{11},y_{12},y_{13}) \right)_{\langle 2 \rangle} \\
             =(y_{11}y_{22}, y_{11}y_{23}, y_{11}y_{33}, y_{12}y_{23}, y_{12}y_{33}, y_{13}y_{23}, y_{22}y_{23}, y_{22}y_{33}, y_{23}^2).
        \end{multlined}
    \end{equation}
    In particular, since all the larger monomials form the diagonal of a $2\times 2$ submatrix of $Y$, we have $\ini(f)\preceq y_{13}y_{23}$ by (\ref{eq:in-P2-symmetric-n=3}). We can thus set
    \[
    f=a_1 y_{13}y_{23} + a_2y_{13}y_{33} +a_3 y_{22}^2 +a_4y_{22}y_{23} + a_5y_{23}^2+a_6y_{23}y_{33}+a_7y_{33}^2,
    \]
    where we remark that $y_{22}y_{33}$ does not appear as a summand due to (\ref{eq:in-P2-symmetric-n=3}).
        Recall from Theorem~\ref{thm:Grobner-symmetric} that a Gr\"obner basis of $P_2(Y)$ in this case is
    \begin{multline*}
        \mathcal{G}\coloneq \{y_{11}y_{22}+y_{12}^2, y_{11}y_{33}+y_{13}^2, y_{11}y_{23}+y_{12}y_{13}, y_{12}y_{33}+y_{13}y_{23}, y_{12}y_{23}+y_{13}y_{22},  y_{22}y_{33}+y_{23}^2, \\
         y_{12}y_{13}y_{22}, y_{12}y_{13}^2, y_{12}^2y_{13}, y_{13}y_{23}^2, y_{13}y_{22}y_{23}, y_{13}^2y_{23}, y_{13}^2y_{22}^2, y_{13}^3y_{22} \}.
    \end{multline*}
    Next we reduce $y_{13}f\in P_2(Y)$ with respect to $\mathcal{G}$:
    \begin{align*}
        y_{13}f
         &=\mathbf{a_1 y_{13}^2y_{23}} + a_2y_{13}^2y_{33} +a_3y_{13} y_{22}^2 +a_4y_{13}y_{22}y_{23} + a_5y_{13}y_{23}^2+a_6y_{13}y_{23}y_{33}+a_7y_{13}y_{33}^2\\
         &\xrightarrow{y_{13}^2y_{23}} \mathbf{a_2y_{13}^2y_{33}} +a_3y_{13} y_{22}^2 +a_4y_{13}y_{22}y_{23} + a_5y_{13}y_{23}^2+a_6y_{13}y_{23}y_{33}+a_7y_{13}y_{33}^2.
    \end{align*}
    If $a_2\neq 0$ or $a_3\neq 0$, then the initial term of the above polynomial is not reducible with respect to $\mathcal{G}$, and thus this forces $a_2=a_3=0$. With that, we continue the reduction:
    \begin{align*}
        & \ \ \ \ \mathbf{a_4y_{13}y_{22}y_{23}} + a_5y_{13}y_{23}^2+a_6y_{13}y_{23}y_{33}+a_7y_{13}y_{33}^2\\
        &\xrightarrow{y_{13}x_{22}y_{23}}  \mathbf{a_5y_{13}y_{23}^2}+a_6y_{13}y_{23}y_{33}+a_7y_{13}y_{33}^2\\
        &\xrightarrow{y_{13}y_{23}^2} \mathbf{a_6y_{13}y_{23}y_{33}}+a_7y_{13}y_{33}^2
    \end{align*}
    Since neither of the monomial summands above are in $\ini(P_2(Y))$, we must have $a_6=a_7=0$. We can thus rewrite $f$ as follows:
    \[
    f=a_1 y_{13}y_{23} + a_4y_{22}y_{23} + a_5y_{23}^2,
    \]
    Next we use the fact that $y_{12}f\in P_2(X)$, and perform reduction with respect to $\mathcal{G}$:
    \begin{align*}
        y_{12}f &= \mathbf{a_1 y_{12}y_{13}y_{23}} + a_4y_{12}y_{22}y_{23} +a_5y_{12}y_{23}^2\\
        &\xrightarrow{y_{12}y_{23}+y_{13}y_{22}} -a_1 y_{13}^2y_{22} + \mathbf{a_4y_{12}y_{22}y_{23} }+a_5y_{12}y_{23}^2\\
        &\xrightarrow{y_{12}y_{23}+y_{13}y_{22}} -a_1 y_{13}^2y_{22} + a_4y_{13}y_{22}^2 +\mathbf{a_5y_{12}y_{23}^2}\\
        &\xrightarrow{y_{12}y_{23}+y_{13}y_{22}} -\mathbf{a_1 y_{13}^2y_{22}} + a_4y_{13}y_{22}^2 -a_5y_{13}y_{22}y_{23}.
    \end{align*}
    With similar arguments, we deduce $a_1=a_4=0$. Thus $f=a_5y_{23}^2$. Finally, we perform reduction on $y_{11}f\in P_2(X)$ with respect to $\mathcal{G}$:
    \[
    y_{11}f = a_5y_{11}y_{23}^2 \xrightarrow{y_{11}y_{23}+y_{12}y_{13}} -a_5y_{12}y_{13}y_{23} \xrightarrow{y_{12}y_{23}+y_{13}y_{22}} a_5 y_{13}^2y_{22},
    \]
    which is not reducible with respect to $\mathcal{G}$ if $a_5\neq 0$. Thus we have $a_5=0$, i.e., $f=0$, a contradiction, as desired.
\end{proof}

\begin{lemma}\label{lem:alpha-symmetric-3}
    Assume that $n\geq 3$. Then for $1 \leq r < s \leq n$, $\alpha\left(\frac{P_2(Y)\colon \q_{rs}(Y)}{P_2(Y)}\right) \geq 3$.
\end{lemma}

\begin{proof}
Due to the symmetry of arguments, we may assume, without loss of generality, that $r = n-1$ and $s = n$. We will prove this result by induction on $n$. If $n=3$, this follows from Lemma~\ref{lemma:n=3}. Now assume that $n\geq 4$ and that 
    \[
    \alpha\left(\frac{P_2(Z)\colon \q_{n-2,n-1}(Z)}{P_2(Z)}\right) \geq 3
    \]
for any generic symmetric matrix $Z$ of size $(n-1)\times (n-1)$.
Assume to the contrary that there exists a homogeneous polynomial $f\in (P_2(Y)\colon \q_{n-1,n}(Y)) \setminus P_2(Y)$ such that $\deg f= 2$. 
Set
    \[
    f=\sum_{x,z\in Y, x\succeq z} a_{xz}xz.
    \]
Since we can replace $f$ with $f-g$ for some $g\in P_2(Y)$, by Theorem~\ref{thm:Grobner-symmetric}, we may assume that
\begin{equation}\label{eq:not-in-P2}
    a_{xz}=0 \quad \text{if $x,z$ form the diagonal of a submatrix of $Y$,}
\end{equation}
and 
\begin{equation}\label{eq:not-in-P2-4-distinct}
     a_{xz}=0 \quad \text{if the row and column indices of $x$ and $z$ form a set of four distinct integers}. 
\end{equation}

Set $\ini(f)=a_{x_0z_0}x_0z_0$. We have the following claim.
\begin{claim}\label{clm:y-z-first-row}
    Neither $x_0$ nor $z_0$ is in the first row of $Y$.
\end{claim}

\begin{proof}[Proof of Claim~\ref{clm:y-z-first-row}]
    Suppose to the contrary that $x_0$ or $z_0$ is in the first row of $Y$. Since $x_0\succeq z_0$, both scenarios imply that $x_0$ is in the first row of $Y$. In fact, we will show that $x_{0}=y_{1k}$ and $z_0=y_{jk}$ for some $1<j<k$ with $k\in \{n-1,n\}$. We first eliminate all other cases. 
\begin{itemize}
    \item \textbf{Case 1:} Asssume that $x_0=y_{11}$. Due to (\ref{eq:not-in-P2}), $z_0$ must be in the first row. We thus have two cases.
    \begin{itemize}
        \item \textbf{Case 1.1:} $z_0=y_{11}$. Then $\ini(y_{11}f)=y_{11}^3\notin \ini(P_2(Y))$ by Theorem~\ref{thm:Grobner-symmetric}, a contradiction.
        \item \textbf{Case 1.2:} $z_0=y_{1i}$ for some $i>1$. Then $\ini(y_{11}f)=y_{11}^2y_{1i}\notin \ini(P_2(Y))$  by Theorem~\ref{thm:Grobner-symmetric}, a contradiction.
    \end{itemize}
    \item \textbf{Case 2:} Assume that $x_0=y_{1k}$ for some $k>1$. We have the following cases.
    \begin{itemize}
        \item \textbf{Case 2.1:} $z_0=y_{jj}$. Since $x_0\succeq z_0$, we have $j>1$. Due to (\ref{eq:not-in-P2}), we have $j\leq k$. Thus $\ini(y_{1i}h)=y_{1i}^2y_{jj}\notin \ini(P_2(Y))$  by Theorem~\ref{thm:Grobner-symmetric}, a contradiction.
        \item \textbf{Case 2.2:} $z_0=y_{1j}$ for some $j>1$. Then $\ini(y_{11}f)=y_{11}y_{1k}y_{1j}\notin \ini(P_2(Y))$  by Theorem~\ref{thm:Grobner-symmetric}, a contradiction.
        \item \textbf{Case 2.3:} $z_0=y_{ji}$ for some $1<j<i\leq n-2$. Due to (\ref{eq:not-in-P2-4-distinct}), we have either $j=k$ or $i=k$, only the latter of which is possible due to (\ref{eq:not-in-P2}).
        Then $\ini(y_{kk}f) = y_{1k}y_{jk}y_{kk} \notin \ini(P_2(Y))$  by Theorem~\ref{thm:Grobner-symmetric}, a contradiction.

    \end{itemize}
\end{itemize}

Therefore we can indeed assume that $f\in P_2(Y)\colon \q_{n-1,n}(Y)$ and $\ini(f) = y_{1k}y_{jk}$ for some $1<j<k$ and $k\in \{n-1,n\}$. We then have $f\q_{n-1,n}(Y)\subseteq  P_2(Y)$. By specializing all variables to 0, except those in the 1st, $j$-th, and $k$-th columns and rows, we obtain
\[
(y_{11},y_{1j},y_{1k}) f \subseteq P_2(Y'), \text{ or equivalently, } f\in P_2(Y')\colon (y_{11},y_{1j},y_{1k})
\]
where 
\[
Y'=\begin{pmatrix}
    y_{11} & y_{1j} & y_{1k}\\
    y_{1j} & y_{jj} & y_{jk}\\
    y_{1k} & y_{jk} & y_{kk}
\end{pmatrix}.
\]
Note that under the new notation, we have $(y_{11},y_{1j},y_{1k}, y_{jj}y_{kk}+y_{jk}^2)=\q_{23}(Y')$. Thus we have
\[
f\in P_2(Y')\colon (y_{11},y_{1j},y_{1k}) = P_2(Y')\colon \q_{23}(Y')
\]
with $\deg f=2$. By Lemma~\ref{lemma:n=3}, (\ref{eq:not-in-P2}), and (\ref{eq:not-in-P2-4-distinct}), we must have $f=0$, a contradiction, as~desired.
\end{proof}
By Claim~\ref{clm:y-z-first-row}, all monomial summands of $f$ are in $\mathbb{K}[\hat{Y}]$. Thus for any $2\leq i\leq j\leq n$ with $i\leq n-2$, by Lemma~\ref{lem:retract}, $y_{ij}f \in P_2(Y) \cap \mathbb{K}[\hat{Y}]=P_2(\hat{Y})$. Therefore we have
\[
f\in P_2(\hat{Y})\colon \q_{n-2,n-1}(\hat{Y}).
\]
By the induction hypothesis, (\ref{eq:not-in-P2}) and (\ref{eq:not-in-P2-4-distinct}), we must have $f=0$, a contradiction, as desired.
\end{proof}

We are finally ready to prove the main theorem.
\begin{theorem}
    Assume that $n\geq 3$. Then $\vv(P_2(Y))=3$.
\end{theorem}
\begin{proof}
Since $\q_{rs}(Y)$, where $1\leq r<s\leq n$, are exactly the minimal primes of $P_2(Y)$ by \cite[Theorem 4.1]{Chau}, from \Cref{lem:alpha-symmetric-3} and \Cref{lem:lower-bound-v-number}, it follows that $\vv(P_2(Y)) \geq 3$. On the other hand, by \cite[Proposition 5.3]{Chau}, $P_2(Y) \colon y_{ij}y_{kk}^2$ is an associated prime of $P_2(Y)$. Therefore, $\vv(P_2(Y)) \leq 3$.
\end{proof}
\section{v-number of permanental ideals of generic hankel matrices}

Throughout this section, let $H$ be an $m\times n$ generic Hankel matrix with $2\leq m\leq n$. In this section, we compute the ${\rm v}$-number of the permanental ideal $P_2(H)$. We write
\begin{equation}\label{equ:Hankel}
    H = \begin{bmatrix} x_1 & x_2 & x_3 & \cdots & x_{n-1} & x_n \\
x_2 & x_3 & x_4 & \cdots & x_n & x_{n+1} \\
\vdots & \vdots & \vdots & \ddots & \vdots & \vdots \\
x_m & x_{m+1} & x_{m+2} & \cdots & x_{m+n-2} & x_{m+n-1}
\end{bmatrix}.
\end{equation}

Consider the lexicographical ordering with respect to the monomial ordering $x_1>x_2>\cdots > x_{m+n-1}$.  A Gr\"obner basis of $P_2(H)$ with respect to this ordering was given in \cite{GGS07}, although it is noteworthy that the result is heavily dependent on the value of $(m,n)$. Below we will employ this Gr\"obner basis at times, with an implicit note that we are using the lex ordering with $x_1>x_2>\cdots > x_{m+n-1}$.

First we handle the cases where $\vv(P_2(H))=1$.

\begin{theorem}\label{thm:Hankel-v-number-1}
    Let $H$ be a generic Hankel matrix as in (\ref{equ:Hankel}). Assume that one of the following holds:
    \begin{enumerate}
        \item $m \geq 3$ and $m+n-1 \geq 9$;
        \item $(m,n)=(3,6)$;
        \item $(m,n)=(4,5)$.
    \end{enumerate}
    Then $\vv(P_2(H)) = 1$.
\end{theorem}
\begin{proof}
    It is known that $P_2(H)$ is not a prime ideal as long as $(m,n)\neq (2,2)$ (see, e.g., \cite[Theorem~4.3]{GGS07}). In particular, under our assumption, $P_2(H)$ is not prime, and thus $\vv(P_2(H))\geq 1$. It now suffices to prove that $\vv(P_2(H))\leq 1$, which we will do case-by-case.
    \begin{enumerate}
        \item Assume that $m \geq 3$ and $m+n-1 \geq 9$. It follows from the proof of \cite[Proposition~5.1 (2)]{GGS07}  that $P_2(H) \colon  x_5 = (x_1,\ldots,x_{m+n-1})$. Hence $\vv(P_2(H)) \leq 1$ by \cref{lem:upper-bound-v-number}, as desired.
        \item Assume that $(m,n)=(3,6)$. We show that $P_2(H)\colon x_5=(x_1,x_2,\dots, x_7)$. For ease of notation, set
        \[
        H=(h_{ij})=\begin{pmatrix}
            x_1 &x_2 &x_3&x_4&x_5&x_6\\
            x_2 &x_3&x_4&x_5&x_6&x_7\\
            x_3&x_4&x_5&x_6&x_7&x_8
        \end{pmatrix}.
        \]
        We have 
        \begin{align*}
            2x_1x_5 &= (x_1x_5+x_3^2)+(x_1x_5+x_2x_4)-(x_2x_4+x_3^2)\\
            &= (h_{11}h_{33}+h_{13}h_{31})+(h_{11}h_{24}+h_{14}h_{21})-(h_{12}h_{23}+h_{13}h_{22})\\
            &\in P_2(H),\\
            2x_2x_5&=(x_2x_5+x_3x_4)+(x_1x_6+x_2x_5)-(x_1x_6+x_3x_4)\\
            &=(h_{12}h_{24}+h_{14}h_{22})+(h_{11}h_{25}+h_{15}h_{21})-(h_{11}h_{34}+h_{14}h_{31})\\
            &\in P_2(H), \\
            2x_3x_5 & = (x_3x_5+x_4^2) + (x_2x_6+x_3x_5) - (x_2x_6+x_4^2) \\
            & = (h_{13}h_{24} + h_{14}h_{23}) + (h_{12}h_{25}+h_{15}h_{22}) - (h_{12}h_{34}+h_{14}h_{32}) \\
            & \in P_2(H),\\
            x_4x_5,x_5^2,x_5x_6&\in P_2(H) \quad \text{due to \cite[Theorem 2.10]{GGS07}},\\
            2x_5x_7&=(x_4x_8+x_5x_7)+(x_5x_7+x_6^2)-(x_4x_8+x_6^2)\\
            &=(h_{23}h_{36}+h_{26}h_{33})+(h_{15}h_{26}+h_{16}h_{25})-(h_{14}h_{36}+h_{16}h_{34})\\
            &\in P_2(H).
        \end{align*}
        Hence $(P_2(H)\colon x_5)\supseteq (x_1,x_2,\dots, x_7)$. Suppose for the sake of contradiction that $(P_2(H)\colon x_5)\supsetneq (x_1,x_2,\dots, x_7)$.     
        Since $x_5\notin P_2(H)$, the colon ideal $P_2(H)\colon x_5$ is a proper ideal, and hence primary to the maximal ideal. Thus $P_2(H)\colon x_5$ contains a power of $x_8$, which we assume to be $x_8^k$ for some $k\geq 1$. We then have $x_5x_8^k\in P_2(H)$. Recall that the following set of polynomials is a Gr\"obner basis of $P_2(H)$ (\cite[Theorem~2.10]{GGS07}):
        \begin{align*}
            \{x_ix_{i+s+t}+x_{i+s}x_{i+t}\mid i\in [1,6], s\in [1,2], t\in [1,5] \text{ with } i+s+t\in [3,8] \}\\
            \cup \{ x_ix_{i+1}, x_i^2, x_6^2\mid i\in [3,5]  \} \cup \{x_2^2x_3, x_6x_7^2, x_2^4,x_7^4\}.
        \end{align*}
        We next reduce $x_5x_8^k$ with respect to this Gr\"obner basis:
        \[
        x_5x_8^k\xrightarrow{x_5x_8+x_6x_7} -x_6x_7x_8^{k-1}.
        \]
        If $k=1$, then the polynomial above, $-x_6x_7$, is not reducible with respect to the above Gr\"obner basis, a contradiction, as desired. Now assume that $k\geq 2$. We then have
        \[
        -x_6x_7x_8^{k-1}\xrightarrow{x_6x_8+x_7^2} x_7^2x_8^{k-2},
        \]
        which is not reducible with respect to the above Gr\"obner basis, a contradiction, as desired.
        \item Assume that $(m,n)=(4,5)$. The permanental ideal in this case can be verified with Macaulay2 to be equal to that in the previous case ($(m,n)=(3,6)$) over the ring of integers $\mathbb{Z}$. Therefore this is the same ideal as the one treated in (2) over any field $\kk$. The result then follows.\qedhere
    \end{enumerate}
\end{proof}

Next is the cases where $\vv(P_2(H))=2$. This class in particular includes $P_2(H)$ where $H$ is a $2\times n$ Hankel matrix, where $n\geq 4$.

\begin{theorem}\label{thm:Hankel-v-number-2}
    Let $H$ be a generic Hankel matrix as in (\ref{equ:Hankel}). Assume that one of the following holds:
    \begin{enumerate}
        \item $m=2$ and $n\geq 4$;
        \item $(m,n)=(3,4)$;
        \item $(m,n)=(3,5)$;
        \item $(m,n)=(4,4)$.
    \end{enumerate}
    Then $\vv(P_2(H))= 2$.
\end{theorem}

\begin{proof}
    \begin{enumerate}
        \item Assume that $m=2$ and $n\geq 4$. It follows from the proof of \cite[Proposition~5.1 (1)]{GGS07} that $P_2(H) \colon x_2x_n = (x_1,x_2,\ldots, x_{n+1})$. Therefore,  $\vv(P_2(H)) \leq 2$ by \cref{lem:upper-bound-v-number}. 
        
        It remains to show that $\vv(P_2(H))\ge 2$. Recall that
        \[
        \ass(P_2(H)) = \{(x_1,x_2,\ldots, x_{n}), \ (x_2,x_3,\ldots, x_{n+1}),\ (x_1,x_2,\ldots, x_{n+1}) \}.
        \]
        By \cref{lem:alpha-smaller-for-colon-of-small-ideals}, we have
        \begin{align*}
            \alpha((P_2(H)\colon (x_1,x_2,\dots, x_n))/ P_2(H)) &\geq \alpha ((P_2(H)\colon x_1)/ P_2(H)),\\
            \alpha((P_2(H)\colon (x_2,x_3,\dots, x_{n+1}))/ P_2(H)) &\geq \alpha ((P_2(H)\colon x_{n+1})/ P_2(H)),\\
            \alpha((P_2(H)\colon (x_1,x_2,\dots, x_{n+1}))/ P_2(H)) &\geq \alpha ((P_2(H)\colon x_1)/ P_2(H)).
        \end{align*}
        By \cref{lem:lower-bound-v-number}, to show $\vv(P_2(H))\geq 2$,   
        it suffices to prove that $P_2(H)\colon x_1$ does not contain a linear form, as the arguments for $P_2(H)\colon x_{n+1}$ would follow by symmetry. Indeed, suppose that $a_1x_1+\cdots +a_{n+1}x_{n+1}\in P_2(H)\colon x_1$, i.e., 
        \[
        \sum_{i=1}^{n+1} a_ix_1x_{i}= a_1x_1^2+\cdots +a_{n+1}x_1x_{n+1} \in P_2(H),
        \]
        for some $a_1,\dots, a_{n+1}\in \kk$. We will show that $a_1=a_2=\cdots = a_{n+1}=0$. 
        We have the following claim.
        \begin{claim}\label{clm:reduction-2xn}
            For each $k\in [1,\lfloor (n+1)/2\rfloor]$, we have
            \[
            \sum_{i=2k-1}^{n+1} a_ix_kx_{i-k+1}\in P_2(H)
            \]
            and $a_{j}=0$ for any $j\in \{2k-1,2k\}$, as long as $j\leq n+1$.
        \end{claim}

        \begin{proof}
            Recall from \cite[Theorem~2.7]{GGS07} that  
            \[
            \{ x_ix_{i+t+1}     +x_{i+1}x_{i+t}\colon i\in [1,n-1], t\in [1,n-i] \} \cup \Omega_1,
            \]
            where $\Omega_1$ is a certain set of polynomials of degree at least 3, is a Gr\"obner basis of $P_2(H)$. As the polynomials of interest will be of degree 2, $\Omega_1$ is of no consequence, and thus we do not specify what it is. 

            We prove the claim by induction on $k$. First consider the base case $k=1$. We already have $\sum_{i=1}^{n+1} a_ix_1x_{i}\in P_2(H)$. Since $x_1^2, x_1x_2 \notin \ini(P_2(H))$, it follows that $a_1 = a_2 = 0$.
This concludes the base case $k=1$. 
            
            By induction, we can assume that $k\geq 2$,
            \[
            \sum_{i=2k-3}^{n+1} a_ix_{k-1}x_{i-k+2}\in P_2(H),
            \] 
            and $a_{2k-3}=a_{2k-2}=0$. We will reduce the polynomial $\sum_{i=2k-3}^{n+1} a_ix_{k-1}x_{i-k+2} = \sum_{i=2k-1}^{n+1} a_ix_{k-1}x_{i-k+2}$ using the above Gr\"obner basis:
            \begin{align*}
                &\sum_{i=2k-1}^{n+1} a_ix_{k-1}x_{i-k+2}= a_{2k-1}x_{k-1}x_{k+1}+ \sum_{i=2k}^{n+1} a_ix_{k-1}x_{i-k+2}\\
                &\xrightarrow{x_{k-1}x_{k+1}+x_k^2} \left( \sum_{i=2k}^{n+1} a_ix_{k-1}x_{i-k+2}\right) - \left( a_{2k-1}x_k^2 \right)\\
                &=\left( \sum_{i=2k}^{n+1} a_ix_{k-1}x_{i-k+2}\right) - \left(\sum_{i=2k-1}^{2k-1} a_{i}x_kx_{i-k+1}. \right)
            \end{align*}
            We note that the above process works regardless of whether $a_{2k-1}$ is $0$ or not. Repeating this process, the polynomial $\sum_{i=2k-1}^{n+1} a_ix_{k-1}x_{i-k+2}$ reduces to $-\sum_{i=2k-1}^{n+1} a_ix_{k}x_{i-k+1}$. In particular, this implies that 
            \[
            \sum_{i=2k-1}^{n+1} a_ix_{k}x_{i-k+1}\in P_2(H),
            \]
            as desired. It remains to show that $a_{2k-1}=a_{2k}=0$, as long as the index makes sense. Indeed, if $a_{2k-1}\neq 0$ or $a_{2k}\neq 0$, then the leading term of $\sum_{i=2k-1}^{n+1} a_ix_{k}x_{i-k+1}$ is either $a_{2k-1}x_{k}^2$ or $a_{2k}x_{k}x_{k+1}$, both of which are not in the initial ideal of $P_2(H)$, a contradiction. Therefore, we have $a_{2k-1}=a_{2k}=0$.  This proves the claim.
        \end{proof}
        
        \cref{clm:reduction-2xn} in particular implies that $a_1=a_2=\cdots = a_{n+1}=0$, as desired. This concludes (1).
        \item Assume that $(m,n)=(3,4)$. It follows from the proof of \cite[Proposition 5.3]{GGS07} that   $P_2(X)\colon x_2x_5 = (x_1,\ldots,x_6)$. Therefore $\vv(P_2(H)) \leq 2$ by \cref{lem:upper-bound-v-number}. It remains to show that $\vv(P_2(H))\geq 2$. Recall from \cite[Theorem~4.3 and Proposition~5.3]{GGS07} that
        \[
        \ass(P_2(H)) = \{(x_1,x_2,\ldots, x_{5}), \ (x_2,x_3,\ldots, x_{6}),\ (x_1,x_2,\ldots, x_{6}) \}.
        \]
        By similar arguments to those in the proof of (1), it suffices to prove that $P_2(H)\colon (x_1,x_2)$ does not contain a linear form. Indeed, suppose that \[
        a_1x_1+\cdots +a_{6}x_{6}\in P_2(H)\colon (x_1,x_2).\]
        We recall from \cite[Proposition~2.9]{GGS07} that
        \begin{multline*}
            \{ x_3^2, \ x_4^2, \ x_1x_5, \  x_2x_4, \  x_2x_6, \  x_3x_6,\  x_1x_3+x_2^2,\ x_1x_4+x_2x_3, \ x_1x_6+x_3x_4, \\
              x_2x_5+x_3x_4, \ x_4x_6+x_5^2,\ x_3x_6+x_4x_5  \} \cup \Omega_2,
        \end{multline*}
        where $\Omega_2$ is a certain set of polynomials of degree at least 3, is a Gr\"obner basis of $P_2(H)$.

        Since $a_1x_1+\cdots + a_6x_6\in P_2(H):(x_1)$, we have
        \[
        a_1x_1^2+a_2x_1x_2+a_3x_1x_3+a_4x_1x_4+a_5x_1x_5+a_6x_1x_6 \in P_2(H).
        \]
        As $x_1^2$ and $x_1x_2$ are not in the initial ideal of $P_2(H)$, we must have $a_1=a_2=0$. Next we reduce the above polynomial above using our Gr\"obner basis:
        \begin{align*}
            a_3x_1x_3+a_4x_1x_4+a_5x_1x_5+a_6x_1x_6 &\xrightarrow{x_1x_3+x_2^2} a_4x_1x_4 + a_5x_1x_5+ a_6x_1x_6 - a_3x_2^2\\
            &\xrightarrow{x_1x_4+x_2x_3}  a_5x_1x_5+ a_6x_1x_6 - a_3x_2^2 - a_4x_2x_3\\
            &\xrightarrow{x_1x_5} a_6x_1x_6 - a_3x_2^2 - a_4x_2x_3\\
            &\xrightarrow{x_1x_6+x_3x_4} - a_3x_2^2 - a_4x_2x_3- a_6x_3x_4.
        \end{align*}
        As none of the three monomials in the above polynomial is in the initial ideal of $P_2(X)$, we must have $a_3=a_4=a_6 = 0$. We therefore have 
        \[
        a_1x_1+\cdots + a_6x_6 = a_5x_5.
        \]
        Now we use the hypothesis that $a_5x_5\in P_2(H)\colon x_2$ to deduce that $a_5x_2x_5\in P_2(H)$, which implies that $a_5=0$ as $x_2x_5$ is not in the initial ideal of $P_2(H)$. To sum up, $P_2(H)\colon (x_1,x_2)$ does not contain a linear form, as desired.
        \item Assume that $(m,n)=(3,5)$. It follows from similar arguments as in the proof of \cref{thm:Hankel-v-number-1} (2) that $P_2(X)\colon x_2x_3 = (x_2,\ldots,x_7)$. Therefore $\vv(P_2(H)) \leq 2$ by \cref{lem:upper-bound-v-number}. It remains to show that $\vv(P_2(H))\geq 2$. Recall from \cite[Proposition~6.4]{GGS07} that
        \[
        \ass(P_2(H)) = \{(x_1,x_2,\ldots, x_{6}), \quad (x_2,x_3,\ldots, x_{7}) \}.
        \]
        By similar arguments to those in the proof of (1), it suffices to prove that $P_2(H)\colon (x_1,x_2)$ does not contain a linear form. This follows from similar arguments to those in the proof of (2).
        \item Assume that $(m,n)=(4,4)$. It follows from the proof of \cite[Proposition 5.4]{GGS07} that   $P_2(X)\colon x_2x_5 = (x_1,\ldots,x_7)$. Therefore $\vv(P_2(H)) \leq 2$ by \cref{lem:upper-bound-v-number}. It remains to show that $\vv(P_2(H))\geq 2$. Recall from \cite[Theorem~4.3 and Proposition~5.4]{GGS07} that
        \[
        \ass(P_2(H)) = \{(x_1,x_2,\ldots, x_{6}), \ (x_2,x_3,\ldots, x_{7}), \ (x_1,x_2,\ldots, x_{7}) \}.
        \]
        By similar arguments to those in the proof of (1), it suffices to prove that $P_2(H)\colon (x_1,x_2)$ does not contain a linear form. This follows from similar arguments to those in the proof of  (2). \qedhere
    \end{enumerate}
\end{proof}

Finally, in the remaining two cases, we have $\vv(P_2(H))=3$.

\begin{theorem}\label{thm:Hankel-v-number-3}
    Let $H$ be a generic Hankel matrix as in (\ref{equ:Hankel}). Assume that one of the following holds:
    \begin{enumerate}
        \item $(m,n)=(2,3)$;
        \item $(m,n)=(3,3)$.
    \end{enumerate}
    Then $\vv(P_2(H))= 3$.
\end{theorem}

\begin{proof}
    \begin{enumerate}
        \item Assume that $(m,n)=(2,3)$. It follows from similar arguments as in the proof of \cref{thm:Hankel-v-number-1} (2) that $P_2(X)\colon x_1x_2x_3 = (x_2,x_3,x_4)$. Therefore $\vv(P_2(H)) \leq 3$ by \cref{lem:upper-bound-v-number}. It remains to show that $\vv(P_2(H))\geq 3$. Recall from \cite[Proposition~6.3]{GGS07} that
        \[
        \ass(P_2(H)) = \{(x_1,x_2,x_3), \quad (x_2,x_3,x_4) \}.
        \]
        By \cref{lem:lower-bound-v-number}, it suffices to show that $\alpha((P_2(H)\colon (x_1,x_2,x_3))/P_2(H))\geq 3$, as the proof for $\alpha((P_2(H)\colon (x_2,x_3,x_4))/P_2(H))\geq 3$ would follow by symmetry. To do this, we employ the method in the proof of \cref{lem:alpha-generic-3}: Prove  by contradiction  that $(P_2(H)\colon (x_1,x_2,x_3))/P_2(H)$ has no quadratic generator. Suppose that
        there exists a quadratic polynomial $f\in P_2(H)\colon (x_1,x_2,x_3)$. It suffices to show that $f\in P_2(H)$. Since we can replace $f$ with $f-g$ for any $g\in P_2(H)$, and $x_1x_3+x_2^2,\ x_1x_4+x_2x_3,\ x_2x_4+x_3^2\in P_2(H)$, we can assume that
        \[
        f= a_1x_1^2+a_{12}x_1x_2+a_2x_2^2+a_{23}x_2x_3+a_3x_3^2+a_{34}x_3x_4+a_4x_4^2.
        \]
        Before proceeding, we recall from \cite[Theorem~2.7]{GGS07} that the Gr\"obner basis of $P_2(H)$~is   
        \[
        \{x_1x_3+x_2^2,\ x_1x_4+x_2x_3,\ x_2x_4+x_3^2,\ x_2^2x_3,\ x_2x_3^2, \ x_2^4,\ x_3^4 \}.
        \]
        We have
        \begin{align*}
            a_1x_1^3+a_{12}x_1^2x_2+a_2x_1x_2^2+a_{23}x_1x_2x_3+a_3x_1x_3^2+a_{34}x_1x_3x_4+a_4x_1x_4^2=x_1f \in P_2(H).
        \end{align*}
        Since $\ini(x_1f)\in \ini(P_2(H))$, we must have $a_1=a_{12}=a_2=0$. Thus
        \[
        f=a_{23}x_2x_3+a_3x_3^2+a_{34}x_3x_4+a_4x_4^2.
        \]        
        We then reduce $x_3f\in P_2(H)$ with respect to the above Gr\"obner basis:
        \begin{align*}
            & \ \ \ \ x_3f= \mathbf{a_{23}x_2x_3^2}+a_3x_3^3+a_{34}x_3^2x_4+a_4x_3x_4^2\\
            &\xrightarrow{x_2x_3^2} a_3x_3^3+a_{34}x_3^2x_4+a_4x_3x_4^2.
        \end{align*}
        Since none of the monomial summand above is in $\ini(P_2(H))$, we must have $a_3=a_{34}=a_4=0$. Thus we have $f=a_{23}x_2x_3$. Now we reduce $x_1f\in P_2(H)$ with respect to the above Gr\"obner basis:
        \[
        x_1f=a_{23}x_1x_2x_3 \xrightarrow{x_1x_3+x_2^2} -a_{23}x_2^3,
        \]
        which must reduce to $0$. Thus we have $a_{23}=0$. In other words, we have $f=0$, a contradiction, as desired.
        \item Assume that $(m,n)=(3,3)$. It follows from the proof of \cite[Proposition~5.2]{GGS07} that   $P_2(X)\colon x_1x_3x_5 = (x_1,\ldots,x_5)$. Therefore $\vv(P_2(H)) \leq 3$ by \cref{lem:upper-bound-v-number}. It remains to show that $\vv(P_2(H))\geq 3$. Recall from \cite[Theorem~4.3 and Proposition~5.2]{GGS07} that
        \[
        \ass(P_2(H)) = \{(x_1,x_2,x_3, x_{4}), \ (x_2,x_3,x_4, x_{5}), \ (x_1,x_2,x_3,x_4, x_{5}) \}.
        \]
        By \cref{lem:lower-bound-v-number}, it suffices to show that $\alpha((P_2(H)\colon (x_1,x_2,x_3,x_4))/P_2(H))\geq 3$, as the proof for $\alpha((P_2(H)\colon (x_2,x_3,x_4,x_5))/P_2(H))\geq 3$ would follow by symmetry, and  $\alpha((P_2(H)\colon (x_1,x_2,x_3,x_4,x_5))/P_2(H))\geq 3$ would follow from \cref{lem:alpha-smaller-for-colon-of-small-ideals}. As the arguments are similar to those in (1), we omit the details. \qedhere
    \end{enumerate}
\end{proof}

We conclude the paper with a natural question:

\begin{question}
    What is the value of $\vv(P_t(X))$, where $t\geq 3$, and $X$ is a generic/generic symmetric/generic Hankel matrix?
\end{question}

\bibliographystyle{amsplain}
\bibliography{refs}

@article {LS00,
    AUTHOR = {Laubenbacher, Reinhard C. and Swanson, Irena},
     TITLE = {Permanental ideals},
   JOURNAL = {J. Symbolic Comput.},
  FJOURNAL = {Journal of Symbolic Computation},
    VOLUME = {30},
      YEAR = {2000},
    NUMBER = {2},
     PAGES = {195--205},
      ISSN = {0747-7171,1095-855X},
   MRCLASS = {13P10 (13B22 13C40 13H10)},
  MRNUMBER = {1777172},
MRREVIEWER = {Jos\'e\ F.\ Andrade},
       DOI = {10.1006/jsco.2000.0363},
       URL = {https://doi.org/10.1006/jsco.2000.0363},
}

@article {GGS07,
    AUTHOR = {Grieco, E. and Guerrieri, A. and Swanson, I.},
     TITLE = {Permanental ideals of {H}ankel matrices},
   JOURNAL = {Abh. Math. Sem. Univ. Hamburg},
  FJOURNAL = {Abhandlungen aus dem Mathematischen Seminar der Universit\"at
              Hamburg},
    VOLUME = {77},
      YEAR = {2007},
     PAGES = {39--58},
      ISSN = {0025-5858,1865-8784},
   MRCLASS = {13P10 (13C40 13F20)},
  MRNUMBER = {2379328},
MRREVIEWER = {Aihua\ Li},
       DOI = {10.1007/BF03173488},
       URL = {https://doi.org/10.1007/BF03173488},
}

@article {Chau,
    AUTHOR = {Chau, Trung},
     TITLE = {Permanental ideals of symmetric matrices},
   JOURNAL = {Proc. Amer. Math. Soc.},
  FJOURNAL = {Proceedings of the American Mathematical Society},
    VOLUME = {154},
      YEAR = {2026},
    NUMBER = {1},
     PAGES = {119--132},
      ISSN = {0002-9939,1088-6826},
   MRCLASS = {13C05 (05E40 13C40 13P10 15A15)},
  MRNUMBER = {5002074},
       DOI = {10.1090/proc/17512},
       URL = {https://doi.org/10.1090/proc/17512},
}

@article{Chau-F,
  author    = {Chau, Trung},
  title     = {The {$F$}-singularities of algebras defined by permanents},
  journal   = {Collectanea Mathematica},
  year      = {2025},
  publisher = {Springer Nature},
  doi       = {10.1007/s13348-025-00494-8},
  url       = {https://link.springer.com/article/10.1007/s13348-025-00494-8}
}

@Article{GRV21,
AUTHOR = {Grisalde, Gonzalo and Reyes, Enrique and Villarreal, Rafael H.},
TITLE = {Induced Matchings and the v-Number of Graded Ideals},
JOURNAL = {Mathematics},
VOLUME = {9},
YEAR = {2021},
NUMBER = {22},
ARTICLE-NUMBER = {},
URL = {https://www.mdpi.com/2227-7390/9/22/2860},
ISSN = {2227-7390},
PAGES={2860},
DOI = {10.3390/math9222860},
}

@article {CSTPV20,
    AUTHOR = {Cooper, Susan M. and Seceleanu, Alexandra and Toh{\u{a}}neanu, {\c{S}}tefan O. and Pinto, Maria Vaz and Villarreal, Rafael H.},
     TITLE = {Generalized minimum distance functions and algebraic
              invariants of {G}eramita ideals},
   JOURNAL = {Adv. in Appl. Math.},
  FJOURNAL = {Advances in Applied Mathematics},
    VOLUME = {112},
      YEAR = {2020},
     PAGES = {101940, 34},
      ISSN = {0196-8858,1090-2074},
   MRCLASS = {13P25 (11T71 13C40 14G50 94B27)},
  MRNUMBER = {4011111},
MRREVIEWER = {C\'icero\ Carvalho},
       DOI = {10.1016/j.aam.2019.101940},
       URL = {https://doi.org/10.1016/j.aam.2019.101940},
}

@article {JV21,
    AUTHOR = {Jaramillo, Delio and Villarreal, Rafael H.},
     TITLE = {The v-number of edge ideals},
   JOURNAL = {J. Combin. Theory Ser. A},
  FJOURNAL = {Journal of Combinatorial Theory. Series A},
    VOLUME = {177},
      YEAR = {2021},
     PAGES = {Paper No. 105310, 35},
      ISSN = {0097-3165,1096-0899},
   MRCLASS = {05E40 (05C65 13A15)},
  MRNUMBER = {4139109},
MRREVIEWER = {Margherita\ Barile},
       DOI = {10.1016/j.jcta.2020.105310},
       URL = {https://doi.org/10.1016/j.jcta.2020.105310},
}

@article {SS22,
    AUTHOR = {Saha, Kamalesh and Sengupta, Indranath},
     TITLE = {The v-number of monomial ideals},
   JOURNAL = {J. Algebraic Combin.},
  FJOURNAL = {Journal of Algebraic Combinatorics. An International Journal},
    VOLUME = {56},
      YEAR = {2022},
    NUMBER = {3},
     PAGES = {903--927},
      ISSN = {0925-9899,1572-9192},
   MRCLASS = {13F55 (05C70 05E40 13A15 13A70)},
  MRNUMBER = {4491066},
MRREVIEWER = {Somayeh\ Bandari},
       DOI = {10.1007/s10801-022-01137-y},
       URL = {https://doi.org/10.1007/s10801-022-01137-y},
}

@article {SK24,
    AUTHOR = {Saha, Kamalesh and Kotal, Nirmal},
     TITLE = {On the v-number of {G}orenstein ideals and {F}robenius powers},
   JOURNAL = {Bull. Malays. Math. Sci. Soc.},
  FJOURNAL = {Bulletin of the Malaysian Mathematical Sciences Society},
    VOLUME = {47},
      YEAR = {2024},
    NUMBER = {6},
     PAGES = {Paper No. 167, 17},
      ISSN = {0126-6705,2180-4206},
   MRCLASS = {13H10 (05E40 13A35 13F20)},
  MRNUMBER = {4795771},
MRREVIEWER = {Mohammad\ Jarrar},
       DOI = {10.1007/s40840-024-01763-8},
       URL = {https://doi.org/10.1007/s40840-024-01763-8},
}

@article {ASS24,
    AUTHOR = {Ambhore, Siddhi Balu and Saha, Kamalesh and Sengupta,
              Indranath},
     TITLE = {The v-number of binomial edge ideals},
   JOURNAL = {Acta Math. Vietnam.},
  FJOURNAL = {Acta Mathematica Vietnamica},
    VOLUME = {49},
      YEAR = {2024},
    NUMBER = {4},
     PAGES = {611--628},
      ISSN = {0251-4184,2315-4144},
   MRCLASS = {13F20 (05E40 13F65)},
  MRNUMBER = {4834446},
MRREVIEWER = {Aryampilly\ V.\ Jayanthan},
       DOI = {10.1007/s40306-024-00540-w},
       URL = {https://doi.org/10.1007/s40306-024-00540-w},
}

@article {DJS25,
    AUTHOR = {Dey, Deblina and Jayanthan, A. V. and Saha, Kamalesh},
     TITLE = {On the {${\rm v}$}-number of binomial edge ideals of some
              classes of graphs},
   JOURNAL = {Internat. J. Algebra Comput.},
  FJOURNAL = {International Journal of Algebra and Computation},
    VOLUME = {35},
      YEAR = {2025},
    NUMBER = {1},
     PAGES = {119--143},
      ISSN = {0218-1967,1793-6500},
   MRCLASS = {13F65 (05C69 05E40 13F55)},
  MRNUMBER = {4869330},
       DOI = {10.1142/S0218196724500607},
       URL = {https://doi.org/10.1142/S0218196724500607},
}

@article {KMT25,
    AUTHOR = {Kataoka, Tatsuya and Muta, Yuji and Terai, Naoki},
     TITLE = {The v-numbers of {S}tanley-{R}eisner ideals from the viewpoint
              of {A}lexander dual complexes},
   JOURNAL = {J. Algebra},
  FJOURNAL = {Journal of Algebra},
    VOLUME = {684},
      YEAR = {2025},
     PAGES = {589--611},
      ISSN = {0021-8693,1090-266X},
   MRCLASS = {13F55 (13H10)},
  MRNUMBER = {4942694},
       DOI = {10.1016/j.jalgebra.2025.07.017},
       URL = {https://doi.org/10.1016/j.jalgebra.2025.07.017},
}

@article {Fic25,
    AUTHOR = {Ficarra, Antonino},
     TITLE = {Simon conjecture and the v-number of monomial
              ideals},
   JOURNAL = {Collect. Math.},
  FJOURNAL = {Collectanea Mathematica},
    VOLUME = {76},
      YEAR = {2025},
    NUMBER = {3},
     PAGES = {477--492},
      ISSN = {0010-0757,2038-4815},
   MRCLASS = {13F20 (05C70 05E40 13F55)},
  MRNUMBER = {4950310},
       DOI = {10.1007/s13348-024-00441-z},
       URL = {https://doi.org/10.1007/s13348-024-00441-z},
}

@article {BM25,
    AUTHOR = {Biswas, Prativa and Mandal, Mousumi},
     TITLE = {A study of {${\rm v}$}-number for some monomial ideals},
   JOURNAL = {Collect. Math.},
  FJOURNAL = {Collectanea Mathematica},
    VOLUME = {76},
      YEAR = {2025},
    NUMBER = {3},
     PAGES = {667--682},
      ISSN = {0010-0757,2038-4815},
   MRCLASS = {05E40 (05C38 05C69 13F20 13F55)},
  MRNUMBER = {4950319},
       DOI = {10.1007/s13348-024-00451-x},
       URL = {https://doi.org/10.1007/s13348-024-00451-x},
}

@article {Kirkup,
    AUTHOR = {Kirkup, George A.},
     TITLE = {Minimal primes over permanental ideals},
   JOURNAL = {Trans. Amer. Math. Soc.},
  FJOURNAL = {Transactions of the American Mathematical Society},
    VOLUME = {360},
      YEAR = {2008},
    NUMBER = {7},
     PAGES = {3751--3770},
      ISSN = {0002-9947,1088-6850},
   MRCLASS = {13C40 (13-04)},
MRREVIEWER = {Irena\ Swanson},
       DOI = {10.1090/S0002-9947-08-04340-7},
       URL = {https://doi.org/10.1090/S0002-9947-08-04340-7},
}

@article {BCMV25,
    AUTHOR = {Ada Boralevi and Enrico Carlini and Mateusz Micha{\l}ek and Emanuele Ventura},
     TITLE = {On the codimension of permanental varieties},
   JOURNAL = {Adv. Math.},
  FJOURNAL = {Advances in Mathematics},
    VOLUME = {461},
      YEAR = {2025},
     PAGES = {Paper No. 110079, 28},
      ISSN = {0001-8708,1090-2082},
   MRCLASS = {14M12 (05E14 05E40 15A15)},
       DOI = {10.1016/j.aim.2024.110079},
       URL = {https://doi.org/10.1016/j.aim.2024.110079},
}

\end{document}